\RequirePackage{fix-cm}
\documentclass[smallextended]{svjour3}       
\smartqed  
\usepackage{graphicx}

\usepackage{amsmath} 
\usepackage{amssymb}  
\usepackage{graphicx}
\usepackage{appendix}
\usepackage{url}
\usepackage{color}

\usepackage{amsthm}
\newtheorem{rem}{Remark}
\theoremstyle{definition} 

\begin{document}

\title{On Maximizing the Distance to a Given Point over an Intersection of Balls II
}


\author{Marius Costandin  
}


\institute{C. Marius \at
              General Digits \\
              \email{costandinmarius@gmail.com}             \\
}

\date{Received: date / Accepted: date}

\maketitle

\begin{abstract}
In this paper the problem of maximizing the distance to a given fixed point over an intersection
of balls is considered. It is known that this problem is NP complete in the general case, since 
any subset sum problem can be solved upon solving a maximization of the distance over an 
intersection of balls to a point inside the convex hull. 

 The general context is: in \cite{funcos1} it is shown that exists 
a polynomial algorithm which always solves the maximization problem if the given point is outside 
 the convex hull of the centers of the balls. Naturally one asks if there is a polynomial algorithm which 
solves the problem for a point inside the convex hull. A conjecture stated in a previous paper, \cite{funcos1} is proved, under 
slightly stronger conditions.  The proven conjecture allows a polynomial 
algorithm for points on the facets of the convex hull and shows that such points share the maximizer 
with all the points in a small enough ball centered at it, thus including points in the interior of the 
convex hull of the ball centers.

%
\end{abstract}

\section{Introduction}
 Let $m>n \in \mathbb{N}$ and $C_k \in \mathbb{R}^n$ for $k \in \{1, \hdots, m\}$ such that any facet of their convex hull does not contain more than $n$ points. 
For a fixed $C_0 \in \mathbb{R}^n$ and $r > 0$ consider 
\begin{align}\label{E1}
\mathcal{Q} &= \bigcap_{k = 1}^m \bar{\mathcal{B}}(C_k, r) \hspace{0.5cm} h(x) = \max_{k \in \{1, \hdots. m\}} \|x - C_k\|^2 - r^2 \hspace{0.5cm} g(x) =  \|x - C_0\|^2 \nonumber \\ \mathcal{P}_{R^2} &= \left\{ x \in \mathbb{R}^n \biggr| \max_{k \in \{1, \hdots. m\}} h(x) - g(x) \leq -R^2 \right\}  \hspace{0.3cm} \mathcal{H}^{\star} = \mathop{\text{argmin}}_{h(x) \leq 1} h(x) - g(x)
\end{align} where $R>0$ and $\bar{\mathcal{B}}(y,R) = \{x \in \mathbb{R}^n | \|x - y\| \leq R \}$ denotes the closed ball of center $y$ and radius $R$. 

The problem studied in this paper is
\begin{align}\label{E2}
\max_{x \in \mathcal{Q}} \|x - C_0\|
\end{align}
The problem (\ref{E2}) is NP complete in general. Noting that $h(x) - g(x)$ is a piecewise linear function, follows that finding an element in $\mathcal{H}^{\star}$ is a convex optimization problem. 

The following results from \cite{funcos1} are reiterated:
\begin{enumerate}
\item The set $\mathcal{Q} = \{x | h(x) \leq 0\} \subseteq \mathcal{P}_{0^2}$

\item If $C_0 \in \text{int}(\text{conv} \{C_1, \hdots, C_m\})$ then the set $\mathcal{H}^{\star}$ has exactly one element $x^{\star}$. This does not depend on the choice of $C_0$ and is the center of the minimum enclosing ball (MEB) of the points $C_1, \hdots, C_m$, see Theorem 2 in \cite{funcos1}. In this case $\mathcal{H}^{\star} \subseteq \mathcal{Q} $ and exists $\underline{R} > 0$ such that $\mathcal{H^{\star}} = \mathcal{P}_{\underline{R}^2}$. The Theorem 1 in  \cite{funcos1} states that
\begin{align}\label{E3}
\max_{x \in \mathcal{Q}} \|x - C_0\| = \min \{R > 0 | \mathcal{P}_{R^2} \subseteq \mathcal{Q}\}
\end{align} Basically, this means that as $R$ increases from $0$ to $\underline{R}$ the set $\mathcal{P}_{R^2}$ evolves from initially containing $Q$ to being included in $Q$. The parameter $R$ for which it first enters $Q$, is actually the maximum distance from $C_0$ to a point in $Q$. The extreme points will be the vertices of the polytope $\mathcal{P}_{R^2}$ last to enter the set $Q$, hence finitely many.  

\item If $C_0 \not\in \text{conv} \{C_1, \hdots, C_m\}$ then the set $\mathcal{P}_{R^2}$ is unbounded, for any $R \geq 0$. The Theorem 1 in \cite{funcos1} states that in this case
\begin{align}
\max_{x \in \mathcal{Q}} \|x - C_0\| = \max \{R > 0 | \mathcal{Q} \cap \mathcal{P}_{R^2} \neq \emptyset\}
\end{align}  Because $\mathcal{Q}$ is bounded and although unbounded $\mathcal{P}_{R^2}$ is shrinking as $R$ increases (being the level sets of $h(x) - g(x)$ one has $\mathcal{P}_{R_1^2} \subseteq \mathcal{P}_{R_2^2}$ for $R_1 \geq R_2$), follows that exists $R_0$ such that $\mathcal{P}_{R^2} \cap \mathcal{Q} = \emptyset$ for all $R > R_0$. Therefore the set $\mathcal{P}_{R^2}$ evolves from initially containing $\mathcal{Q}$ for $R = 0$ to not having common elements for $R > R_0$. The largest parameter $R$ for which the set $\mathcal{P}_{R^2}$ has common elements to $\mathcal{Q}$ is actually the maximum distance from $C_0$ to a point in $Q$. In this case it is proven that there is always an unique extreme point, see \cite{funcos1}. For this case it is possible to compute in polynomial time the maximum distance and the maximizer, as showed in \cite{funcos1}.

\item Finally, if $C_0 \in \partial \text{conv}(C_1, \hdots, C_m)$ then from \cite{funcos1} one has:
\begin{align}
\max_{x \in \mathcal{Q}} \|x - C_0\| = \underline{R} = \max \{R > 0 | \mathcal{P}_{R^2} \neq \emptyset\} = \|y - C_0\| \hspace{0.3cm} \forall y \in \partial \mathcal{Q} \cap \mathcal{H}^{\star}
\end{align} For this case it is conjectured in \cite{funcos1} that the number of extreme points is either one, either an (uncountable) infinity. 
\end{enumerate} 


\section{Geometry Results}
We begin this section with a proof for the above conjecture. 
\begin{lemma}\label{L1}
Let $C_0 \in \partial \text{conv}(C_1, \hdots,C_m)$ with $C_0 = \sum_{k=1}^p \alpha_k \cdot C_{\sigma_k}$ with $\alpha_k > 0$, $\sum_{k=1}^p \alpha_k = 1$, $p \leq n$ and $\sigma_k \in \{1, \hdots, m\}$ with $ \sigma_k \neq \sigma_j$ for $k \neq j$.  Then if
\begin{enumerate}
\item $p = n$ then the number of solution to the problem (\ref{E2}) is exactly one or two.
\item $p < n$ then the number of solution to the problem (\ref{E2}) is exactly one or an uncountable infinity.
\end{enumerate}
\end{lemma}
\begin{proof} The set $\mathcal{P}_{R^2}$ is the intersection of the following sets: \begin{align}\label{E6b}
&\|x - C_k\|^2 - r^2 - \|x - C_0\|^2 \leq  -R^2 \iff \nonumber \\
& 2\cdot (C_0 - C_k)^T\cdot x  + \|C_k\|^2 - r^2 + R^2 - \|C_0\|^2 \leq 0
\end{align} for all $k \in \{1, \hdots, m\}$

Assume w.l.o.g that $C_0 \in \text{conv}(C_1, \hdots, C_{p})$ with $p \leq n$, hence exist the real numbers $\alpha_k > 0$ with $\sum_{k=1}^p \alpha_k = 1$ and $C_0 = \sum_{k=1}^p \alpha_k \cdot C_k$. The corresponding facets of $\mathcal{P}_{R^2}$  are:

\begin{align}\label{E16a}
\begin{cases}
2 \cdot (C_0 - C_1)^T \cdot x + \|C_1\|^2 \leq \|C_0\|^2 - R^2 + r^2\\
\vdots \\
2 \cdot (C_0 - C_p)^T \cdot x + \|C_p\|^2 \leq \|C_0\|^2 - R^2 + r^2\\
\end{cases}
\end{align} It is shown that exists $R$ which meets the above system with equality. Indeed, by setting equality and subtracting the last line from the first (p-1) lines, one gets: 
\begin{align}\label{E17a}
\begin{cases}
2 \cdot (C_p - C_1)^T \cdot x - \|C_p\|^2 + \|C_1\|^2 = 0\\
\vdots \\
2 \cdot (C_p - C_{p-1})^T \cdot x - \|C_p\|^2 + \|C_{p-1}\|^2 = 0\\
2 \cdot (C_0 - C_p)^T \cdot x + \|C_p\|^2 = \|C_0\|^2 - R^2 + r^2\\
\end{cases}
\end{align} The first $p-1$ lines in (\ref{E17a}) are always linearily independent because the points $C_1, \hdots, C_p$ are affinely independent. Let $v_p, \hdots, v_n \in \{C_p - C_1, \hdots, C_p - C_{p-1}\}^{\perp}$ be orthogonal with $\|v_p\| = \hdots = \|v_n\| = 1$. Add the equations $v_k^T \cdot (x - C_p) = t_k$ for $t_k \in \mathbb{R}$ for all $k \in \{p, \hdots, n\}$. One finds $x(t_p, \hdots, t_n) = x(0, \hdots, 0) + \sum_{k=1}^p t_k \cdot v_k$ unique for a given $t_p, \hdots, t_n$ which meets them. Replacing this in the last equation, together with 
$C_0 = \sum_{k=1}^p \alpha_k \cdot C_k$ where $\sum_{k=1}^{p}\alpha_k = 1, \alpha_k > 0 $ one has 
\begin{align}\label{E18a}
&2 \cdot  \sum_{k=1}^p \alpha_k\cdot (C_k - C_p)^T \cdot x + \|C_p\|^2 = \|C_0\|^2 - R^2 + r^2 \nonumber \\
&\sum_{k=1}^{p-1} \alpha_k \cdot \left( \|C_k\|^2 - \|C_p\|^2\right) + \|C_p\|^2 = \sum_{k=1}^p \alpha_k \cdot \|C_k\|^2 = \|C_0\|^2 -R^2 + r^2 
\nonumber \\
& = \sum_{k=1}^p \alpha_k \cdot (\|C_k\|^2 - r^2) -\|C_0\|^2 = - R^2 
\end{align} Since $\exists x_0 \in \mathcal{Q}$ (i.e $\mathcal{Q}$ is not empty) one has $\|x_0 - C_k\| \leq r$ for all $k \in \{1, \hdots, m\}$, therefore $\|C_k\|^2 -r^2 \leq -\|x_0\|^2 + 2 \cdot x_0^T \cdot C_k$
and (\ref{E18a}) becomes
\begin{align}
-R^2 = \sum_{k=1}^p \alpha_k \cdot (\|C_k\|^2 - r^2) -\|C_0\|^2 &\leq -\|x_0\|^2 + 2 \cdot x_0^T \cdot \sum_{k=1}^p\alpha_k\cdot C_k - \|C_0\|^2 \nonumber \\
&=  -\|x_0 - C_0\|^2 \leq 0
\end{align} hence (\ref{E18a}) is feasible.  Let 
\begin{align}
\underline{R} = \sqrt{-\sum_{k=1}^p \alpha_k \cdot (\|C_k\|^2 - r^2) +\|C_0\|^2}
\end{align} Therefore for $R = \underline{R}$ one has $x = x(0) + \text{span}\left\{ \{C_p - C_1, \hdots, C_p - C_{p-1}\}^{\perp}\right\}$ meets (\ref{E16a}) with equality. Let $y$ in the feasible set of (\ref{E16a}) for $R = \underline{R}$ with $y = x + \beta \cdot u$ for $\beta \in \mathbb{R}$ and $u \in \mathbb{R}^n$. 

Therefore
\begin{align}\label{E21a}
\begin{cases}
2 \cdot (C_0 - C_1)^T \cdot (x(t) + \beta \cdot u) + \|C_1\|^2 \leq \|C_0\|^2 - \underline{R}^2 + r^2\\
\vdots \\
2 \cdot (C_0 - C_p)^T \cdot (x(t) + \beta \cdot u) + \|C_p\|^2 \leq \|C_0\|^2 - \underline{R}^2 + r^2\\
\end{cases}
\end{align} but since $x$ meets (\ref{E17a}) with equality, follows that (\ref{E21a}) is equivalent to
\begin{align}\label{E22a}
\begin{cases}
2 \cdot (C_0 - C_1)^T \cdot u  \leq 0\\
\vdots \\
2 \cdot (C_0 - C_p)^T \cdot u  \leq 0\\
\end{cases} 
\end{align} Multiply each line in (\ref{E22a}) with $\alpha_k > 0$ and add them to obtain the following for each $k \in \{1, \hdots, p\}$
\begin{align}
2 \cdot \left( C_0 - \sum_{k=1}^n \alpha_k \cdot C_k \right)^T \cdot u = 0 \ \Rightarrow \ \alpha_k \cdot (C_0 - C_k)^T \cdot u = 0
\end{align} This is motivated by the fact that if the sum of $p$ non-positive numbers is zero, then each number must be zero. Therefore the product $\alpha_k \cdot (C_0 - C_k)^T \cdot u $ must be zero. Since $\alpha_k > 0$, this means:
\begin{align}
\begin{cases}
2 \cdot (C_0 - C_1)^T \cdot u  = 0\\
\vdots \\
2 \cdot (C_0 - C_p)^T \cdot u  = 0\\
\end{cases} \iff \begin{cases}
2 \cdot (C_p - C_1)^T \cdot u  = 0\\
\vdots \\
2 \cdot (C_p - C_{p-1})^T \cdot u  = 0\\
2 \cdot (C_0 - C_p)^T \cdot u = 0
\end{cases}
\end{align} from the first $p-1$ equations follows that $u \in \{C_p - C_1, \hdots, C_p - C_{p-1}\}^{\perp}$ hence $u \in \text{span}\left\{ \{C_p - C_1, \hdots, C_p - C_{p-1}\}^{\perp}\right\}$ and as such $y \in x(0, \hdots, 0) + \text{span}\left\{ \{C_p - C_1, \hdots, C_p - C_{p-1}\}^{\perp}\right\}$. 

Recall from (\ref{E6b}) the structure of $\mathcal{P}_{\underline{R}^2}$. The intersection of the half-spaces formed with the points $C_k$ for $k > p$ is unbounded since $C_0 \not\in \text{conv} \{C_k | k \in \{p+1, \hdots, m\}\}$ while the intersection of the half-spaces formed with the points $\{C_1, \hdots, C_p\}$ has the form $x_0 + \text{span} \left\{ \{C_p - C_1, \hdots, C_p - C_{p-1}\}^{\perp} \right\}$. Let us denote by $\mathcal{P}_{R^2}^{-}$ the intersection of the half-spaces formed with the points $C_k$ for $k > p$ and by $\mathcal{P}_{R^2}^{0}$ the intersection of the half-spaces formed with the points $C_k$ for $k \leq p$. As such one has $\mathcal{P}_{R^2} = \mathcal{P}_{R^2}^{-} \cap \mathcal{P}_{R^2}^{0}$.

Two cases can be distinguished :
\begin{enumerate}
\item If $p = n$ then the intersection of the half-spaces formed with the points $\{C_1, \hdots, C_{p = n}\}$ is an axis since the dimension of the linear space \\ $\text{span} \left\{ \{C_n - C_1, \hdots, C_n - C_{n-1}\}^{\perp} \right\}$ is one. Hence $\mathcal{P}_{\underline{R}^2} \cap \partial\mathcal{Q}$ contains exactly one point or exactly two points. Indeed, 
\begin{enumerate}
\item if $\mathcal{P}_{\underline{R}^2}^{-} \cap \mathcal{P}_{\underline{R}^2}^{0} \cap \text{int}(\mathcal{Q}) \neq \emptyset $ then $\mathcal{P}_{\underline{R}^2} \cap \partial \mathcal{Q} \subseteq \mathcal{P}_{\underline{R}^2}^{0} \cap \partial \mathcal{Q}$ which has at most two points since it is the intersection of an axis with the boundary of the set $\mathcal{Q}$. 
\item otherwise, if $\mathcal{P}_{\underline{R}^2}^{-} \cap \mathcal{P}_{\underline{R}^2}^{0} \cap \text{int}(\mathcal{Q}) = \emptyset $ then exists $R_0 \leq \underline{R}$ such that for all $R > R_0$ one has $\mathcal{P}_{R^2} \cap \partial \mathcal{Q} = \emptyset$. Assume that exists $x_0 \neq x_1 \in \mathcal{P}_{R_0^2} \cap \partial \mathcal{Q} $ hence exists $x_3 = \frac{x_1 + x_2}{2} \in \text{int} (\mathcal{Q}) \cap \mathcal{P}_{R_0^2}$ hence exists $R>R_0$ with $\mathcal{P}_{R^2} \cap \partial \mathcal{Q} \neq \emptyset$, which is a contradiction. Therefore is this situation as well the maximum number of points on $\mathcal{P}_{R^2} \cap \partial \mathcal{Q}$ is exactly one. 
\end{enumerate}
\item Otherwise, if $p < n$ then as above one proves
\begin{enumerate}
\item if $\mathcal{P}_{\underline{R}^2}^{-} \cap \mathcal{P}_{\underline{R}^2}^{0} \cap \text{int}(\mathcal{Q}) \neq \emptyset $ then $\mathcal{P}_{\underline{R}^2} \cap \partial \mathcal{Q} = \left(\mathcal{P}_{\underline{R}^2}^{0} \cap \partial \mathcal{Q}\right) \cap \mathcal{P}_{\underline{R}^2}^{-} $ For this case it can be shown that the resulting intersection has an uncountable number of points. 
\item otherwise, if $\mathcal{P}_{\underline{R}^2}^{-} \cap \mathcal{P}_{\underline{R}^2}^{0} \cap \text{int}(\mathcal{Q}) = \emptyset $ then using the same reasoning as in the previous case one concludes that the maximum number of points on $\mathcal{P}_{R^2} \cap \partial \mathcal{Q}$ is exactly one. 
\end{enumerate}
\end{enumerate}
\end{proof}

\begin{corollary}\label{C1}
In the particular case in which $\mathcal{Q} \subseteq \text{int}(\text{conv}(C_1, \hdots, C_m))$ for $C_0 \in \partial \text{conv}(C_1, \hdots, C_m)$ with $C_0 = \sum_{k=1}^n \alpha_k \cdot C_{\sigma_k}$ with $\alpha_k > 0$, $\sum_{k=1}^n \alpha_k = 1$ and $\sigma_k \in \{1, \hdots, m\}$ with $ \sigma_k \neq \sigma_j$ for $k \neq j$ the number of solutions to the problem 
\begin{align}
\max_{x\in \mathcal{Q}} \|x - C_0\|
\end{align} is exactly one. 
\end{corollary}
Note here the the difference to Lemma \ref{L1} stays in $p = n$. 
\begin{proof}
Indeed, from Lemma \ref{L1} the number of solution is either exactly one either two. However, the two solutions would lie on the intersection of an axis perpendicular on the hyperplane formed by the points $C_{\sigma_1}, \hdots, C_{\sigma_n}$ with the boundary of the set $\mathcal{Q}$. Since this hyperplane does not intersect the set $\mathcal{Q}$ and since $C_0$ belongs to this hyperplane follows that one of these two points is more distant to $C_0$ then the other one. 
\end{proof}

In the following a remark is given related to the above proof:
\begin{rem} Let $x$ be a solution to (\ref{E17a}). Then after a reorganization of the terms in each equation, one has
\begin{align}
\begin{cases}
\|x - C_1\|^2 = \|x - C_0\|^2 -\underline{R}^2 + r^2\\
\vdots \\
\|x - C_p\|^2 = \|x - C_0\|^2 -\underline{R}^2 + r^2\\
\end{cases} \Rightarrow \|x - C_1\|^2 = \hdots = \|x - C_p\|^2
\end{align} meaning that $x$ is equidistant to the points whom $C_0$ is a convex combination of. In particular, for $p = n$ follows that the center of the ball determined by the points $C_1, \hdots, C_n$ (with the certer forced to lie in the hyperplane formed by the points) is also a solution. This means that in this case, the axis which the most distant point to $C_0$ belongs to, is an axis which passes through the center of a ball determined by the points $C_1, \hdots, C_n$ and orthogonal on the hyperplane determined by these points. 
\end{rem}

In the following we shall prove that in the conditions of the Corollary \ref{C1}, the solution to (\ref{E2}) is a vertex of $\mathcal{Q}$. For this we first give the following lemma:

\begin{lemma}\label{L2}
 If the point $C_0 \in \text{int}(\text{conv}(C_1, \hdots, C_m))$ then the farthest point to it in $\mathcal{Q}$ is a corner of $\mathcal{Q}$.
\end{lemma}

\begin{proof}
From (\ref{E3}) follows that $x^{\star}$ a solution to (\ref{E2}) is one of the vertexes of $\mathcal{P}_{(R^{\star})^2}$ last to enter the set $\mathcal{Q}$, where $R^{\star} = \max_{x \in \mathcal{Q}} \|x - C_0\|$ . That is $x^{\star} \in \mathcal{P}_{(R^{\star})^2} \cap \partial \mathcal{Q}$. Assume w.l.o.g that the vertex $x^{\star}$ is the intersection of the following facets of $\mathcal{P}_{(R^{\star})^2}$
\begin{align}\label{E18b}
&\begin{cases}
2\cdot (C_0 - C_1)^T \cdot x + \|C_1\|^2 = \|C_0\|^2 + r^2 - (R^{\star})^2 \\
\vdots \\
2\cdot (C_0 - C_n)^T \cdot x + \|C_n\|^2 = \|C_0\|^2 + r^2 - (R^{\star})^2
\end{cases} \iff \nonumber \\
&
\begin{cases}
\|x - C_1\|^2 - r^2 - \|x - C_0\|^2 = -(R^{\star})^2 \\
\vdots \\
\|x - C_n\|^2 - r^2 - \|x - C_0\|^2 = -(R^{\star})^2
\end{cases}
\end{align} but since $\|x^{\star} - C_0\| = R^{\star}$ follows from (\ref{E18b}) that
\begin{align}
\begin{cases}
\|x^{\star} - C_1\|^2 - r^2 = 0 \\
\vdots \\
\|x^{\star} - C_n\|^2 - r^2 = 0
\end{cases}
\end{align} that is, $x^{\star}$ is a corner of the intersection of balls $\mathcal{Q}$ being on the intersection of at least $n$ spheres. 
\end{proof} Before giving the main result of this section we give a small technical lemma to be used later. This lemma is used to show that if at one moment a point $C_1$, from a group of points, is the farthest to $y$, then letting $y$ slide on an axis to reach another point $z$ to whom $C_1$ is no longer the farthest from the group of points, then $C_1$ will never be the farthest to any points on that axis going in the same direction. 

\begin{lemma} \label{L3}
Let $z,y,C_1,C_2 \in \mathbb{R}^n$ with $\| y - C_1\| = \| y - C_2\| $. Assume, without loss of generality, that $\|z - C1\|^2 \geq \|z - C_2\|^2 $ then
\[
\|y + t  (z-y) - C_1\|^2 \geq \| y + t  (z-y) - C_2 \|^2, \hspace{1cm} \forall t \geq 0. 
\]
\end{lemma}
\begin{proof}
Let 
\[
h(t) = \|y + t  (z-y) - C_1\|^2 - \|y + t  (z-y) - C_2\|^2.
\] 
From the identity above, it  can be  seen that $h(t)$ is a polynomial of degree at most $1$ in $t$. Since
$\|y -C_1\|=\|y -C_2\|$ gives $h(0) = 0$ and $\|z -C_1\|\ge \|z -C_2\| $ gives $h(1) \ge h(0) = 0$, it follows that $h(t)$ is a non-decreasing first order polynomial in $t$ and therefore
$$
 h(t) \ge 0=h(0), \;\; \forall t\ge 0,
$$
which completes the proof. 
\end{proof}

Finally we give the result 

\begin{theorem}\label{T1}
If $\mathcal{Q} \subseteq \text{int}(\text{conv}(C_1, \hdots, C_m))$ and $C_0 \in \partial \text{conv}(C_1, \hdots, C_m)$, $C_0 = \sum_{k=1}^n \alpha_k \cdot C_{\sigma_k}$ with $\alpha_k > 0$, $\sum_{k=1}^n \alpha_k = 1$, $\sigma_k \in \{1, \hdots, m\}$ with $ \sigma_k \neq \sigma_j$ for $k \neq j$ then number of solutions to the problem 
\begin{align}
\max_{x\in \mathcal{Q}} \|x - C_0\|
\end{align} is exactly one and the solution is a vertex of $\mathcal{Q}$.
\end{theorem}
\begin{proof}
For the uniqueness of the solution see Corollary \ref{C1}. For the fact that the unique solution is a vertex of $\mathcal{Q}$, let $\{v^{\star} \} = \mathop{\text{argmax}}_{x \in \mathcal{Q}}\|x - C_0\|$ and consider the segment $\mathcal{S} = \{v^{\star} + t \cdot (C_0 - v^{\star}) | t \in [0,1]\}$ which connects the vertex $v^{\star}$ of $\mathcal{Q}$ with the point $C_0$. Consider the points of $\mathcal{S} \cap \text{int}(\text{conv}(C_1, \hdots, C_m)) = \mathcal{S} \setminus \{C_0\}$. According to Lemma \ref{L2} the farthest points to these points in $\mathcal{Q}$ are among the vertexes of $\mathcal{Q}$.
%
%

 According to Lemma \ref{L3} exists $\epsilon > 0$ and $u^{\star} \in \{\text{vertexes of }\mathcal{Q}\}$ such that for all $y \in \mathcal{B}(C_0, \epsilon) \cap \mathcal{S} \setminus C_0$ one has $u^{\star} \in \mathop{\text{argmax}}_{x \in \mathcal{Q}} \| x - y\|^2$. Therefore, let $D_0 \in \mathcal{B}(C_0, \epsilon) \cap \mathcal{S} \setminus C_0$ and the following are true: 
\begin{enumerate}
\item The points from the segment opened at $C_0$, $\mathcal{D} = \{ D_0 + t \cdot (C_0 - D_0) | t \in [0,1)\}$ share a common vertex of $\mathcal{Q}$ as solution to problem (\ref{E2}), i.e 
\begin{align}
\exists u^{\star} \in \bigcap_{y \in \mathcal{D}} \mathop{\text{argmax}}_{x \in \mathcal{Q}} \|x - y\|^2 
\end{align} and $u^{\star}$ is a vertex of $\mathcal{Q}$.  
\item Any point from the semi-axis $\mathcal{E} = \{C_0 + t \cdot \frac{C_0 - D_0}{\|C_0 - D_0\|}| t \in [0,\infty)\}$ has $v^{\star}$ as the solution to the problem (\ref{E2}) i.e 
\begin{align}
\{v^{\star}\} = \bigcap_{y \in \mathcal{E}} \mathop{\text{argmax}}_{x \in \mathcal{Q}} \|x - y\|^2 
\end{align} where recall that $\{v^{\star} \} = \mathop{\text{argmax}}_{x \in \mathcal{Q}}\|x - C_0\|$. This can be easily seen by noting that $\mathcal{Q} \subseteq \bar{\mathcal{B}}(C_0, \|C_0 - v^{\star}\|) \subseteq \bar{\mathcal{B}}(y, \|C_0 - v^{\star}\| + \|y - C_0\|)$ and $\|y - v^{\star}\| = \|C_0 - v^{\star}\| + \|y - C_0\|$ for all $y \in \mathcal{E}$.  
\end{enumerate}

In the following we ought to prove that $v^{\star} = u^{\star} \in \{ \text{vertexes of } \mathcal{Q}\}$

Let $\mathcal{F} = \mathcal{D} \cup \mathcal{E}$ and the function $\zeta : \mathcal{F} \to \mathbb{R}$ with $\zeta(y) = \max_{x \in \mathcal{Q}} \|x - y\|$ for any $y \in \mathcal{F}$. Note that for any $y \in \mathcal{F}$ one has $\zeta(y)= \begin{cases} \|y - u^{\star}\|, y \in \mathcal{D}\\ \|y - v^{\star}\|, y \in \mathcal{E}\end{cases}$. 

Since the function $\zeta(\cdot)$ is continuous follows that exists $z \in \mathcal{F}$ such that $\zeta(z) = \|z - u^{\star}\| = \|z - v^{\star}\|$. Assuming that $z \in \mathcal{D}$ follows that $v^{\star} \in \mathop{\text{argmax}}_{x \in \mathcal{Q}} \|x - z\|$ hence $v^{\star} \in \{ \text{vertexes of } \mathcal{Q}\}$ since $z \in \mathcal{D} \subseteq \text{int}(\text{conv}(C_1, \hdots, C_m))$. Otherwise, if $z \in \mathcal{E}$ follows that $u^{\star} \in \mathop{\text{argmax}}_{x \in \mathcal{Q}} \|x - z\| = \{v^{\star}\} $ since $v^{\star}$ is the only solution to (\ref{E2}) for $z \in \mathcal{E}$. This, again, leads to the statement $v^{\star} = u^{\star} \in \{ \text{ vertexes of } \mathcal{Q} \} $.  
\end{proof}

The following remark gives a small note on the complexity needed for applying Theorem \ref{T1}. 
\begin{remark} [Complexity analysis]
Before moving to the last result from this section, it is worth saying that Theorem \ref{T1} can be used to compute the farthest point in an intersection of balls to a given fixed point meeting its requirements, in a polynomial number of steps. Indeed, it just shows that one needs to apply the theory presented in \cite{funcos} to obtain the maximizer as either an intersection of an axis with the boundary of $\mathcal{Q}$ either as a point in an intersection of convex sets. It is obvious that obtaining the axis and the convex sets requires a polynomial number of operations. 
\end{remark}

Finally we give a small lemma at the end of this section:
\begin{lemma} \label{L4}
Let $x^{\star}, C_1, \hdots, C_{n+1} \in \mathbb{R}^n$ distinct with 
\begin{enumerate}
\item $x^{\star} \not\in \text{conv}(C_1, \hdots, C_{n+1})$
\item Exists $\alpha_k > 0$ such that $C_{n+1} - x^{\star} = \sum_{k=1}^n \alpha_k \cdot (C_k - x^{\star})$
\item $\|x^{\star} - C_k\| = r$ for all $k \in \{1, \hdots, n+1\}$.
\end{enumerate} then 
\begin{align}
\{x^{\star} \} = \mathop{\text{argmax}}_{x \in \bigcap_{k=1}^n \bar{\mathcal{B}}(C_k,r)} \|x - C_{n+1}\| 
\end{align}
\end{lemma}
\begin{proof} One way to prove the above is by observing that $\bigcap_{k=1}^n \bar{\mathcal{B}}(C_k,r)$ is an intersection of equal radii balls and $C_{n+1}$ is outside of the convex combination of the balls centers. It can be proven using the above that the maximizer is a vertex. Since only two vertices exist and the second is closer to $C_{n+1}$ the conclusion follows.  
\end{proof}

\section{Application: Subset Sum Problem}
Let $n \in \mathbb{N}$ and consider $S \in \mathbb{R}^n$ and $T \in \mathbb{R}$. The associated subset sum problem, SSP(S,T) asks it exists $x \in \{0,1\}^{n}$ such that $x^T\cdot S = T$. For this, similar to \cite{sahni}, consider the optimization problem for $\beta > 0$:
\begin{align}
\max x^T\cdot(x - 1_{n \times 1}) + \beta \cdot S^T \cdot x \hspace{0.5cm} \text{s.t} \ \ \   x \in \begin{cases} S^T\cdot x \leq T\\
0 \leq x_i \leq 1 \hspace{0.3cm} \forall i \in \{1, \hdots, n\}
\end{cases} 
\end{align} Let the feasible set be denoted by $\mathcal{P} = \{x \in \mathbb{R}^n| S^T\cdot x \leq T, 
0 \leq x_i \leq 1 \hspace{0.3cm} \forall i \in \{1, \hdots, n\} \}$. 
\begin{remark}\label{R1}
It is easy to see that the objective function is always smaller than or equal to $\beta \cdot T$. In fact the objective function reaches the value $\beta \cdot T$ if and only if the SSP(S,T) has a solution. 
\end{remark}

Note that the objective function can be rewritten as
\begin{align}
x^T\cdot x + \left(\beta \cdot S - 1_{n \times 1} \right)^T \cdot x &= \left\| x - \frac{1_{n\times 1} - \beta \cdot S}{2} \right\|^2 - \left\| \frac{1_{n\times 1} - \beta \cdot S}{2} \right\|^2 \nonumber \\
& = \|x - C_0\|^2 - \|C_0\|^2
\end{align} with obvious definition for $C_0$. Since $C_0$ does not depend on $x$, we shall consider the optimization problem:
\begin{align}\label{E25}
\max_{x \in \mathcal{P}} \|x - C_0\|^2
\end{align} The problem (\ref{E25}) is a distance maximization over a polytope. Indeed $\mathcal{P}$ is the intersection of the unit hypercube with the halfspace $\{ x | S^T\cdot x \leq T\}$. Any maximizer shall be located in a corner of the polytope $\mathcal{P}$. 

In this section we shall substitute the set $\mathcal{P}$ with an intersection of balls with equal radii (ball polytopes) that preserve the corners of $\mathcal{P}$ if these are also corners of the unit hypercube. We shall prove that for the chosen intersection of balls, if the SSP(S,T) problem has a solution then it is also a solution to the maximization over the intersection of balls. 

\subsection{Construction of the intersection of balls associated to SSP(S,T)}
Here we use a similar construction to the one presented in \cite{funcos1}. As such, let $\mathcal{H}$ denote the unit hypercube, and consider the ball $\mathcal{B}\left( \frac{1}{2} \cdot 1_{n \times 1}, \frac{\sqrt{n}}{2} \right)$. 
For the facet $\{x | x^T \cdot e_k \geq 0\}$ of the hyper-cube, let $C_{k+} = \frac{1}{2}\cdot 1_{n \times 1} + d \cdot e_k$, while for the facet $\{x | x^T \cdot e_k \leq 1\}$ we choose $C_{k-} = \frac{1}{2} \cdot 1_{n \times 1} - d \cdot e_k$ where $d \geq \underline{d} \geq \frac{\sqrt{n}}{2}$ with $\left(d + \frac{1}{2} \right)^2 + \left( \frac{n}{4} - \frac{1}{4}\right) = r^2$ and $\underline{d}$ is explained later. With this choice of parameters one has 
\begin{align}\label{E26}
&\partial \bar{\mathcal{B}}(C_{k+}, r) \cap \partial \bar{\mathcal{B}}\left( \frac{1}{2} \cdot 1_{n \times 1}, \frac{\sqrt{n}}{2} \right) = \{x | x^T \cdot e_k \geq 0\} \cap \partial \bar{\mathcal{B}}\left( \frac{1}{2} \cdot 1_{n \times 1}, \frac{\sqrt{n}}{2} \right) \nonumber \\
&\partial \bar{\mathcal{B}}(C_{k-}, r) \cap \partial \bar{\mathcal{B}}\left( \frac{1}{2} \cdot 1_{n \times 1}, \frac{\sqrt{n}}{2} \right) = \{x | x^T \cdot e_k \leq 1\} \cap \partial \bar{\mathcal{B}}\left( \frac{1}{2} \cdot 1_{n \times 1}, \frac{\sqrt{n}}{2} \right) 
\end{align} where $e_k$ is the $k$'th column of the unit matrix in $\mathbb{R}^n$. Next, let 
\begin{align}
\mathcal{U}_r = \bigcap_{k=1}^n \bar{\mathcal{B}}(C_{k+},r) \cap & \bar{\mathcal{B}}(C_{k-},r) \Rightarrow \nonumber \\
\mathcal{U}_r \cap \partial  \bar{\mathcal{B}}\left( \frac{1}{2} \cdot 1_{n \times 1}, \frac{\sqrt{n}}{2} \right) &= \mathcal{H} \cap \partial  \bar{\mathcal{B}}\left( \frac{1}{2} \cdot 1_{n \times 1}, \frac{\sqrt{n}}{2} \right)
\end{align} therefore the intersection of balls $\mathcal{U}_r$ has the same corners as the hyper-cube $\mathcal{H}$. 

Next, under the assumption that $\{x | S^T \cdot x = T\} \cap \mathcal{H} \neq \emptyset$, let $P_s$ be the projection of $C = \frac{1}{2} \cdot 1_{n \times 1}$ on the hyper-plane $\{x | S^T \cdot x = T\}$ and $C_s = P_s - d_s \cdot \frac{S}{\|S\|}$ where $d_s^2 + \left( \frac{n}{4} - \|C - P_s\|^2\right) = r^2$. With this choice of parameters one has
\begin{align}\label{E28}
\partial \bar{\mathcal{B}}(C_s, r) \cap \partial  \bar{\mathcal{B}}\left( \frac{1}{2} \cdot 1_{n \times 1}, \frac{\sqrt{n}}{2} \right) = \{x | S^T \cdot x = T\} \cap \partial  \bar{\mathcal{B}}\left( \frac{1}{2} \cdot 1_{n \times 1}, \frac{\sqrt{n}}{2} \right)
\end{align} Finally let 
\begin{align}\label{E29}
\mathcal{Q}_r = \mathcal{U}_r \cap \bar{\mathcal{B}}(C_s, r)
\end{align}  Choose $\underline{d}$ such that $\mathcal{Q}_r \subseteq \text{int}(\text{conv}(C_{1\pm}, \hdots, C_{n\pm},C_s))$. 
We note the following:

\begin{remark}\label{R2} 
Note that as $r \to \infty$ one has $\mathcal{Q}_r \to \mathcal{P}$. 
\end{remark}

\begin{remark}\label{R3}
One can easily remark that if the SSP(S,T) has a solution then it is among the corners of $\mathcal{Q}_r$ for any $r$ meeting the above.  
\end{remark}

\subsection{A solution to the SSP(S,T)}
Next we prove that 

\begin{lemma}
if the SSP(S,T) has a solution $x^{\star}$ then for $\|C_0 - C\| \geq \frac{\sqrt{n}}{2}$ $x^{\star}$ is also a solution to the maximization problem:
\begin{align}
\max_{x \in \mathcal{Q}_r} \|x - C_0\|^2 \ \text{for} \ C_0 = C - \frac{\beta}{2} \cdot S \in \text{int} (\text{conv}(C_{1\pm}, \hdots, C_{n\pm},C_s))
\end{align}
\end{lemma}
\begin{proof}
 Indeed since $C_0$ is on the segment $[C,C_s]$ (and therefore in the interior of the convex hull of the balls centers) from Remark \ref{R2} follows that exists $r_0 > 0$ such that $x^{\star} \in \mathop{\text{argmax}}_{x \in \mathcal{Q}_r} \| x - C_0\|$ for all $r \geq r_0$. This means that $\mathcal{Q}_r \subseteq \bar{\mathcal{B}}(C_0,\|x^{\star} - C_0\|)$. 

Now, let $r \leq r_0$. This will bring the points $C_{k\pm}$ and $C_s$ closer to $C$ since $r$ has to meet the criteria from the previous subsection see (\ref{E26}, \ref{E28}). We let $r$ have any value such that $C_0 \in (C, C_s)$ with $C_0$ remaining fixed. 

Since $\bar{\mathcal{B}}(C_s,r) \cap \bar{\mathcal{B}}\left(C,\frac{\sqrt{n}}{2} \right) \subseteq \bar{\mathcal{B}}(C_0,\|x^{\star} - C_0\|) \cap \bar{\mathcal{B}}\left(C,\frac{\sqrt{n}}{2} \right) $, and $\mathcal{Q}_r \subseteq \bar{\mathcal{B}}(C_s,r) \cap \bar{\mathcal{B}}\left(C,\frac{\sqrt{n}}{2} \right) $ follows that $\mathcal{Q}_r \subseteq  \bar{\mathcal{B}}(C_0,\|x^{\star} - C_0\|)$. Finally, since $x^{\star}$ is a corner of the hyper-cube being a solution to the SSP(S,T) follows that $x^{\star} \in  \partial \mathcal{Q}_r$ for any $r$, see Remark \ref{R3}, hence 
\begin{align}
x^{\star} \in \mathop{\text{argmax}}_{x \in \mathcal{Q}_r} \|x - C_0\|
\end{align}
\end{proof}

As such, in the following let $C_0$ be fixed meeting $\|C_0 - C\| \geq \frac{\sqrt{n}}{2}$ and we shall study the problem
\begin{align}
\max_{x\in \mathcal{Q}_r} \|x - C_0\|
\end{align} for any $r$ fixed, which allows $C_0 \in (C,C_s)$ where by $(C,C_s)$ we denote the open segment starting at $C$ and ending at $C_s$. 
If the SSP(S,T) has a solution, then 
\begin{align}\label{E34a}
\max_{x \in \mathcal{Q}_r} \|x - C_0\|^2 = \|C_0 - P_s\|^2 + \left( \frac{n}{4} - \|C - P_s\|^2 \right) =: R_0^2
\end{align} for any $r$ as in the previous subsection such that $C_0 \in (C,C_s)$. Let $r$ be fixed with this property, and let $x^{\star}$ denote the unique solution to the SSP(S,T). Then $x^{\star}$ is a vertex of $\mathcal{P}$ and is also a vertex of $\mathcal{Q}_r$. Construct $\mathcal{P}_{R^2}$ as in (\ref{E1}), a family of polytopes indexed after $R>0$. It follows from Theorem 1 in \cite{funcos1} that $x^{\star} $ is a vertex of $\mathcal{P}_{R_0^2}$ and $\mathcal{P}_{R_0^2} \subseteq \mathcal{Q}_r$, this being the first polytope in the family to enter the set $\mathcal{Q}_r$.  

It follows that in order to test the existence of $x^{\star}$ one just has to assert if $\mathcal{P}_{R_0^2} \subseteq \mathcal{Q}_r$ and if $\mathcal{P}_{R_0^2} \cap \partial \mathcal{Q}_r \neq \emptyset$. For this we do the following:

\textbf{Alternative problem} 

Since $x^{\star}$ is reportedly the unique maximizer of $\max_{x \in \mathcal{Q}_r} \|x - C_0\|$ follows that exists $\epsilon > 0$ such that 
\begin{align}
\{x^{\star} \} = \mathop{\text{argmax}}_{x \in \mathcal{Q}_r} \|x - y\|^2 \hspace{0.5cm} \forall y \in \mathcal{B}(C_0,\epsilon) 
\end{align} 

In order to find $x^{\star}$, if it exists, \textit{randomly} choose $n + 1$ points  $C_{0,p} \in \mathcal{B}(C_0,\epsilon) \cap \text{int}(\text{conv}(C_{1\pm}, \hdots, C_{n\pm},C_s))$ for all $p \in \{1, \hdots, n+1\}$ such that $C_0 \in \text{conv}(C_{0,1}, \hdots, C_{0,n+1})$ and consider the problems:

\begin{align}\label{E35}
\mathop{\text{argmax}}_{x \in \mathcal{Q}_r} \|x - C_{0,p}\|
\end{align}  For (\ref{E35}) form as in (\ref{E1}) the family of polytopes \\
$\mathcal{P}_{R^2,p} = \{x \in \mathbb{R}^n| \max_{k \in \{1\pm, \hdots, n\pm,s\}} \|x - C_k\|^2 - r^2 - \|x - C_{0,p}\|^2 \leq -R^2\}$

From Theorem 1 in \cite{funcos1} follows that exists $R_{0,p}$ such that $\mathcal{P}_{R_{0,p}^2,p} \subseteq \mathcal{Q}_r$ and $\{x^{\star} \}= \mathcal{P}_{R_{0,p}^2,p} \cap \partial \mathcal{Q}_r$ hence $R_{0,p} = \max_{x \in \mathcal{Q}_r} \|x - C_{0,p}\| = \|x^{\star} - C_{0,p}\|$. Therefore
\begin{align}\label{E38a}
R_{0,p} = \|x^{\star} - C_{0,p}\| = \|x^{\star} - C_0 + C_0 -C_{0,p}\| \leq R_0 + \epsilon \nonumber \\
R_0 = \|x^{\star} - C_0\| = \|x^{\star} -C_{0,p} + C_{0,p} - C_0\| \leq R_{0,p} + \epsilon \Rightarrow R_{0,p} \geq R_0 - \epsilon
\end{align} hence $R_{0,p} \in [R_0 - \epsilon, R_0 + \epsilon]$. 

It is known that for each $p$ one has $\mathcal{P}_{R_{0,p}^2,p} \subseteq \mathcal{Q}_r$. However, finding $R_{0,p}$ is hard in general, since deciding if $\mathcal{P}_{\rho^2,p} \subseteq \mathcal{Q}_r$ is equivalent with saying that 
\begin{align}
\max_{x \in \mathcal{P}_{\rho^2,p}} \|x - C_i\| \leq r \hspace{0.5cm} \forall i \in \{1\pm, \hdots, n\pm,s\}
\end{align} each of these problems are a distance maximization over a polytope and we do not have a polynomial algorithm for them. Of course, one can try to replace the polytope $\mathcal{P}_{\rho^2,p}$ with an intersection of balls $\mathcal{Q}_{\rho,p}$ as $\mathcal{Q}_{r}$ was obtained from $\mathcal{P}$ in (\ref{E29}). Unfortunately, $\mathcal{P}_{R_{0,p}^2,p} \subseteq \mathcal{Q}_r$ does not imply $\mathcal{Q}_{R_{0,p},p} \subseteq \mathcal{Q}_r$. That is, in general it is possible that the smalles $\rho$ for which $\mathcal{Q}_{\rho,p} \subseteq \mathcal{Q}_r$ is still larger than $R_{0,p}$. Although $x^{\star} \in \mathcal{Q}_{R_{0,p},p}$ it is possible in general to have $y \in \mathcal{Q}_{R_{0,p},p} $ with $ y \not \in \mathcal{Q}_r$. 

It makes sense therefore, to attempt to "trim" the sets $\mathcal{Q}_{R_{0,p},p}$. One can easily see that $x^{\star} \in \bigcap_{p = 1}^{n+1} \mathcal{Q}_{R_{0,p},p} \subseteq \mathcal{Q}_{R_{0,p},p}$ for all $p$. Even more 
\begin{align}
x^{\star} \in \bigcap_{p = 1}^{n+1} \mathcal{Q}_{R_{0,\underline{p}},p} \subseteq \mathcal{Q}_{R_{0,\underline{p}},p}  \hspace{0.5cm} \forall p \in \{1, \hdots, n+1\}
\end{align} where $R_{0,\underline{p}} = \min \{R_{0,p} | p \in \{1, \hdots, n+1\}\}$. As such, define

\begin{align}
\mathcal{T}_{\rho_1, \dots, \rho_{n+1}} := \bigcap_{p = 1}^{n+1} \mathcal{Q}_{\rho_p,p} \hspace{0.5cm} \mathcal{T}_{\rho} := \mathcal{T}_{\rho, \dots, \rho}
\end{align} Note that $x^{\star} \in \mathcal{T}_{R_{0, \underline{p}},p}$. Next we shall focus in the following on the problem:
\begin{align}
\rho^{\star} = \min \{ \rho | \mathcal{T}_{\rho} \subseteq \mathcal{Q}_r\}
\end{align}

\textbf{ A proper definition of $\mathcal{Q}_{\rho,p}$}

For each facet $k \in \{1, \hdots, 2\cdot n + 1\}$ of $\mathcal{P}_{\rho^2,p}$ let $P_{k,\rho,p}$ be the projection of $C = \frac{1}{2} \cdot 1_{n \times 1}$ on the facet and $C_{k,\rho,p} = P_{k,\rho,p} - d_{k,\rho,p} \cdot \frac{v_{k,p}}{\|v_{k,p}\|}$ where $v_{k,p}$ is the normal vector to the facet (note that $\rho$ is irreleant for this, since $\rho$ is just a translation of the facet) and $d_{k,\rho,p}^2 + \left( \frac{n}{4} - \|C - P_{k,\rho,p}\|^2\right) = r^2_{\rho,p} = r^2$. Here, for simplicity, we consider these balls to have the same radius as the initial balls.  This condition assures

\begin{align}\label{E38}
&\partial \mathcal{B}(C_{k,\rho,p}, r) \cap \partial \mathcal{B}\left( \frac{1}{2} \cdot 1_{n \times 1}, \frac{\sqrt{n}}{2}\right) \nonumber \\
& = \{x | v_{k,p}^T \cdot (x - P_{k,\rho,p}) = 0\} \cap \partial \mathcal{B}\left( \frac{1}{2} \cdot 1_{n \times 1}, \frac{\sqrt{n}}{2}\right)
\end{align} Furthermore, let $r$ large enough such that 
\begin{align}\label{E39}
\{C_{1\pm}, \hdots, C_{n\pm}, C_s \} \not\in \text{int}(\text{conv}\{C_{k, \rho, p} |k \in \overline{1, \hdots, 2 \cdot n + 1}, p \in \overline{1, \hdots, n+1} \})
\end{align}  and define 
\begin{align}
\mathcal{Q}_{\rho,p} = \bigcap_{k = 1}^{2\cdot n + 1} \bar{\mathcal{B}}(C_{k,\rho,p}, r) 
\end{align} 
Since $C_{0,p}$ were chosen randomly, we can apply Theorem \ref{T1} with probability one to obtain for the problem $\max_{ x\in \mathcal{T}_{\rho}} \|x - C_i\|$ for all $ i \in \{1\pm, \hdots, n\pm, s\}$ hence define
\begin{align}\label{E45a}
\{ x^{\star}_i(\rho_1, \hdots, \rho_{n+1}) \} := \mathop{\text{argmax}}_{x \in \mathcal{T}_{\rho_1, \hdots, \rho_{n+1}}} \|x - C_i\|
\end{align} Note that for any given $\rho_1, \hdots, \rho_{n+1}$ one can compute $x^{\star}_i(\rho_1, \hdots, \rho_{n+1})$ in polynomial time using the above results from the Section: Geometry Results. 

\begin{remark} The probability one is due to the fact that the set of points not allowing the stated results (edges of the convex hull of the points $C_{1\pm}, \hdots, C_{n\pm},C_s$) has zero measure in $R^{n}$, hence a random selection would almost surely not pick them.
\end{remark}
We are now able to state the main theorem of this section:
\begin{theorem} \label{T2}
If the $SSP(S,T)$ has a unique solution $x^{\star}$ then exists $\epsilon_0 > 0$ such that for all $\epsilon_0 \geq \epsilon > 0$ by choosing randomly $n+1$ points inside the closed ball $\bar{\mathcal{B}}(C_0,\epsilon)$ such that $C_0$ is in their convex hull, one has with probability one that exists $R_{0,p} \in [R_0-\epsilon, R_0 + \epsilon]$ for $p \in \{1, \hdots, n+1\}$ such that 

\begin{align}
x^{\star} \in \{  x^{\star}_i(R_{0,1}, \hdots, R_{0,n+1})| i \in \{1\pm ,\hdots, n\pm,s\}  
\end{align} where $x^{\star}_i(\rho_1, \hdots, \rho_{n+1})$ is given by (\ref{E45a}) and $R_0$ is given by (\ref{E34a}). 
\end{theorem} 
\begin{proof}
Assume w.l.o.g that $\|x^{\star} - C_{1+}\| = r$  then we shall prove that $x^{\star} = x^{\star}_{1+} (R_{0,1}, \hdots, R_{0,n+1})$. Consider the points $C_{1+, R_{0,1},1}, \hdots, C_{1+, R_{0,n+1},n+1}$ formed each as presented above. Recall that each point is formed using the point $C_{1+}$ and a "disturbance" of the point $C_0$. Because these "disturbances" of $C_0$, a.k.a $C_{0,p}$, are chosen such that $C_0$ is in their convex hull follows that we can choose $n$ out of them (w.l.o.g the first $n$) such that $\exists \alpha_p > 0$ with 
\begin{align}
C_{1+} - C_0 = \sum_{p = 1}^n \alpha_p \cdot (C_{1+, R_{0,p},p}- C_0)
\end{align}   hence exists $\beta_p > 0$ with 
\begin{align}
C_{1+} - x^{\star} =  \sum_{p = 1}^n \beta_p \cdot (C_{1+, R_{0,p},p}- x^{\star})
\end{align} It is known that $x^{\star} \in \mathcal{Q}_r$ and $x^{\star} \in \mathcal{Q}_{R_{0,p},p}$ hence as assumed $\|x^{\star} - C_{1+}\| = r$ i.e is on the facet of the ball polytope $\mathcal{Q}_r$ generated by the point $C_{1+}$. It can be shown that $x^{\star}$ belong to the same facet of the ball polytopes $\mathcal{Q}_{R_{0,p},p}$ (these facets are disturbances of the same facet and coincide if $\epsilon \to 0$) hence $\|x^{\star} - C_{1+, R_{0,p}, p}\| = r$. Since $C_{1+},C_{1+, R_{0,1},1}, \hdots, C_{1+, R_{0,n},n} \in \partial \mathcal{B}(x^{\star},r)$ and \\ 
$x^{\star} \not\in \text{conv}(C_{1+},C_{1+, R_{0,1},1}, \hdots, C_{1+, R_{0,n},n})$ one can apply Lemma \ref{L4} to conclude that $\bigcap_{k=1}^n \bar{\mathcal{B}}(C_{1+, R_{0,p},p}) \subseteq \bar{\mathcal{B}}(C_{1+}, r)$ hence 
\begin{align}
\mathcal{T}_{R_{0,1}, \hdots, R_{0,n+1}} \subseteq \bigcap_{k=1}^n \bar{\mathcal{B}}(C_{1+, R_{0,p},p}) \subseteq \bar{\mathcal{B}}(C_{1+}, r) 
\end{align} 

Furthermore, because $x^{\star} \in \mathcal{T}_{R_{0,1}, \hdots, R_{0,n+1}} \cap \partial \mathcal{B}(C_{1+}, r) $ follows that 
\begin{align}
x^{\star} = \mathop{\text{argmax}}_{x \in \mathcal{T}_{R_{0,1}, \hdots, R_{0,n+1}}} \|x - C_{1+}\|
\end{align} hence $x^{\star} = x^{\star}_{1+}(R_{0,1}, \hdots, R_{0,n+1})$ because the maximizer is unique. 
\end{proof}

\begin{remark}
The above theorem allows one to compute the maximizer $x^{\star}$ if $R_{0,p}$ are given. From (\ref{E38a}) follows that $R_{0,p} \in [R_0-\epsilon, R_0 + \epsilon]$ with $R_0$ being given by (\ref{E34a}). However, the above method cannot be used to solve the SSP because the values of $R_{0,p}$ being not known have to be taken each from their respective interval. This leads to an exponential number of problems to be solved. For this reason we propose an easier problem:
\begin{align}\label{E51}
x^{\star}_{i}(\rho) := \mathop{\text{argmax}}_{x \in \mathcal{T}_{\rho}} \|x - C_i\| \hspace{0.5cm} \forall i \in \{1\pm, \hdots, n\pm, s\}
\end{align} and naturally ask if its solution enjoys the similar properties as those ensured by Theorem \ref{T2}.
\end{remark}

For the problem (\ref{E51}) we give the following result 

\begin{theorem}  \label{T3}
For any $\rho \in [R_{0} - \epsilon, R_0 + \epsilon]$ exists $\delta > 0$ such that 
\begin{align}
 \|C_i - x^{\star}_i(\rho)\| - \delta \leq \|C_i - x^{\star}_i(R_{0,1}, \hdots, R_{0,n+1})\|  \leq \|C_i - x^{\star}_i(\rho)\| + \delta
\end{align} for any $ i\in \{1\pm ,\hdots, n\pm,s\}$
\end{theorem}
\begin{proof}
 The set $\mathcal{T}_{\rho}$ is a perturbation of the set $\mathcal{T}_{R_{0,1}, \hdots, R_{0,n+1}}$. That is, the intersecting balls forming $\mathcal{T}_{\rho}$ are the exact balls whom intersection form $\mathcal{T}_{R_{0,1}, \hdots, R_{0,n+1}}$ with the centers translated by an amount less than $2 \cdot \epsilon$ and the same radius. As such one can say that exists $\delta > 0$ such that $v_{\rho} \in \mathcal{B}(v_{R_{0,1}, \hdots, R_{0,n+1}}, \delta)$ where $v_{\rho}$ is a vertex of $\mathcal{T}_{\rho}$ and  $v_{R_{0,1}, \hdots, R_{0,n+1}}$ is a vertex of $\mathcal{T}_{R_{0,1}, \hdots, R_{0,n+1}}$. Because $x^{\star}_i(\rho)$, the solution to (\ref{E51}) is a vertex of $\mathcal{T}_{\rho}$, follows that exists $v_{R_{0,1}, \hdots, R_{0,n+1}}$ a vertex of $\mathcal{T}_{R_{0,1}, \hdots, R_{0,n+1}}$ with $\|x^{\star}_i(\rho) - v_{R_{0,1}, \hdots, R_{0,n+1}} \| \leq \delta$ hence 
\begin{align}
\|C_i - x^{\star}_i(\rho)\| &= \|C_i - v_{R_{0,1}, \hdots, R_{0,n+1}} + v_{R_{0,1}, \hdots, R_{0,n+1}} - x^{\star}_i(\rho)\| \nonumber \\
& \leq \|C_i - v_{R_{0,1}, \hdots, R_{0,n+1}}\| + \|v_{R_{0,1}, \hdots, R_{0,n+1}} - x^{\star}_i(\rho)\|  \nonumber \\
& \leq \|C_i - x^{\star}_i(R_{0,1}, \hdots, R_{0,n+1})\| + \delta
\end{align} Because it also exists a vertex $v_{\rho}$ of $\mathcal{T}_{\rho}$ in the ball $\mathcal{B}(x^{\star}_i(R_{0,1}, \hdots, R_{0,n+1}), \delta)$ follows
 \begin{align}
\|C_i - x^{\star}_i(R_{0,1}, \hdots, R_{0,n+1})\| &= \|C_i - v_{\rho} + v_{\rho} - x^{\star}_i(R_{0,1}, \hdots, R_{0,n+1})\| \nonumber \\
& \leq \|C_i - v_{\rho}\| + \|v_{\rho} - x^{\star}_i(R_{0,1}, \hdots, R_{0,n+1})\|  \nonumber \\
& \leq \|C_i - x^{\star}_i(\rho) \| + \delta
\end{align}  From here, the conclusion easily follows.
\end{proof}

\begin{remark}
Unfortunately, in the above theorem we cannot give precise bounds on $\delta$, the amount with which the vertices of $\mathcal{T}_{\rho}$ are off to the vertices of $\mathcal{T}_{R_{0,1}, \hdots, R_{0,n+1}}$. This should be investigated in a future work. For the moment, they might depend on the distance $C_0$ has the points $C_i$ (the centers of the balls forming $\mathcal{Q}_r$) for $i \in \{1\pm, \hdots, n\pm,s\}$ since this also influences the angles of the facets. 
\end{remark}

\section{Conclusion}

In this paper we have presented results concerning the maximization of the
distance to a given point over an intersection of balls. In particular, we have 
shown that if the given point is on a facet of the convex hull boundary of the intersection 
of balls, then the maximizer is unique as long as the actual intersection is 
included in the convex hull. It is also shown that the maximizer is a 
vertex for the given context. These results prove a conjecture previously stated on a 
previous research paper \cite{funcos1}. 

 The results are then applied to the Subset Sum Problem (SSP). Here it is shown that
 the subset sum has a solution if and only if the maximum distance over an intersection 
 of balls to a certain point has a predefined expected value. Unfortunately, the point is always
 in the interior of the convex hull of the balls centers. This therefore, does not allow the 
 application of the polynomial algorithm presented in \cite{funcos1}. 
 
 A SSP with a single solution is then analyzed with the presented theory.







%
%



\end{document}